\documentclass[11pt, draft]{amsart}
\usepackage{amssymb, amstext, amscd, amsmath, amssymb}
\usepackage{mathtools, xypic, paralist, color, dsfont, rotating, bbm}
\usepackage{verbatim}
\usepackage{enumerate,enumitem}
\usepackage{relsize} 
\usepackage{float}
\usepackage{setspace}

\usepackage{tikz}

\floatstyle{boxed} 
\restylefloat{figure}

\numberwithin{equation}{section}
\usepackage{quoting}
\quotingsetup{vskip=.1in}
\quotingsetup{leftmargin=.17in}
\quotingsetup{rightmargin=.17in}

\let\OLDthebibliography\thebibliography
\renewcommand\thebibliography[1]{
  \OLDthebibliography{#1}
  \setlength{\parskip}{0pt}
  \setlength{\itemsep}{2pt plus 0.5ex}
}

\usepackage[a4paper]{geometry}
\makeatletter
\def\@cite#1#2{{\m@th\upshape\bfseries%
[{#1\if@tempswa{\m@th\upshape\mdseries, #2}\fi}]}}
\makeatother
\theoremstyle{plain}
\newtheorem{theorem}{Theorem}[section]
\newtheorem{corollary}[theorem]{Corollary}
\newtheorem{proposition}[theorem]{Proposition}
\newtheorem{lemma}[theorem]{Lemma}
\theoremstyle{definition}
\newtheorem{definition}[theorem]{Definition}
\newtheorem{example}[theorem]{Example}

\newtheorem{remark}[theorem]{Remark}

\newtheorem{problem}[theorem]{Problem}
\newtheorem*{acknow}{Acknowledgements}

\theoremstyle{remark}

\mathtoolsset{centercolon}
  \newcommand{\A}{{\mathcal{A}}}
  \newcommand{\B}{{\mathcal{B}}}

  \newcommand{\F}{{\mathcal{F}}}
  
\renewcommand{\H}{{\mathcal{H}}}
  \newcommand{\I}{{\mathcal{I}}}
  \newcommand{\J}{{\mathcal{J}}}

  \newcommand{\T}{{\mathcal{T}}}

  \newcommand{\W}{{\mathcal{W}}}

\def\al{\alpha}

\def\De{\Delta}
\def\de{\delta}

\def\la{\lambda}

\def\om{\omega}

\newcommand{\bC}{\mathbb{C}}

\newcommand{\bN}{\mathbb{N}}

\newcommand{\fl}{{\mathfrak{l}}}
\newcommand{\fJ}{{\mathfrak{J}}}

\newcommand{\foral}{\text{ for all }}

\newcommand{\eL}{\bar{L}}
\newcommand{\bTl}{\bar{\T}_{\lambda}}
\newcommand{\bTu}{\bar{\T}}

\newcommand{\loplus}{\mathlarger{\mathlarger{\oplus}}}
\newcommand{\bt}{\bar{t}}
\newcommand{\bl}{\bar{\lambda}}
\newcommand{\blp}{\bar{\lambda}^{+}}

\newcommand{\eel}{\bar{\ell}^2(P)} 
 
\newcommand{\enl}{\bar{l}}

\newcommand{\ca}{\mathrm{C}^*}

\newcommand{\cenv}{\mathrm{C}^*_{\textup{env}}}

\newcommand{\ol}{\overline}

\newcommand{\ad}{\operatorname{ad}}

\newcommand{\cmax}{\mathrm{C}^*_{\textup{max}}}

\newcommand{\id}{{\operatorname{id}}}

\newcommand{\cspan}{\overline{\operatorname{span}}}

\newtheorem*{theorem*}{Theorem}
\newtheorem*{corollary*}{Corollary}
\newtheorem*{proposition*}{Proposition}
\newtheorem*{lemma*}{Lemma}
\newtheorem*{remark*}{Remark}
\newtheorem*{definition*}{Definition}
\newtheorem*{problem*}{Problem}

\usepackage[normalem]{ulem}
\usepackage{xstring}

\def\dep{\ol{\de}^+}

\def\couni{C}

\def\iotenv{\iota_{\textup{env}}}
\def\delenv{\delta_{\textup{env}}}

\begin{document}

\title[Enhanced left regular representation]
{Fell's absorption principle for semigroup operator algebras}
\date{March 12, 2023.}
\author[E.G. Katsoulis]{Elias G. Katsoulis}
\address{Department of Mathematics\\ East Carolina University\\ Greenville\\ NC 27858\\USA}
\email{katsoulise@ecu.edu}



\begin{abstract}
Fell's absorption principle states that the left regular representation of a group absorbs any unitary representation of the group when tensored with it. In a weakened form, this result carries over to the left regular representation of a right LCM submonoid of a group and its Nica covariant isometric representations but it fails if the semigroup does not satisfy independence. In this paper we explain how to extend Fell's absorption principle to an arbitrary submonoid $P$ of a group $G$ by using an enhanced version of the left regular representation. Li's semigroup $\ca$-algebra $\ca_s(P)$ and its representations appear naturally in our context. Using the enhanced left regular representation, we not only provide a very concrete presentation for the reduced object for $\ca_s(P)$ but we also resolve open problems and obtain very transparent proofs of earlier results. In particular, we address the non-selfadjoint theory and we show that the non-selfadjoint object attached to the enhanced left regular representation coincides with that of the left regular representation. We obtain a non-selfadjoint version of Fell's absorption principle involving the tensor algebra of a semigroup and we use it to improve recent results of Clouatre and Dor-On on the residual finite dimensionality of certain $\ca$-algebras associated with such tensor algebras. As another application, we give yet another proof for the existence of a $\ca$-algebra which is co-universal for equivariant, Li-covariant representations of a submonoid $P$ of a group $G$.
\end{abstract}

\maketitle




\section{Introduction}

Let $G$ be a discrete group and let $ L:= \{L_s\}_{s \in G}$
be its left regular representation, i.e., $L_s(\delta_t) = \delta_{st}$, $s, t \in G$, where $\{\delta_t\}_{t \in G}$ is the canonical basis of $\ell^2(G)$. If $U: = \{ U_s\}_{s \in G}$ is any unitary representation of $G$ on Hilbert space $\H$, then \textit{Fell's absorption principle} states that the representations $L\otimes I$ and $L\otimes U$ are unitarily equivalent and so the map
\[
\ca_{r} (G) \ni L_s \longmapsto L_s\otimes U_s \in B(\ell^2(G) \otimes \H)
\]
extends to a $*$-isomorphism between the generated $\ca$-algebras. 

Let now $P\subseteq G$ be a submonoid. Since its left regular representation on $\ell^2(P)$, denoted again as $L$, is now an isometric representation, one is naturally led to the study of isometric representations of $P$ on Hilbert space (which coincide with the unitary representations if $P$ happens to be a group). If $V = \{V_p\}_{p\in P}$ is such an isometric representation of $P$, one does not expect the representations $L\otimes I$ and $L\otimes V$ to be unitarily equivalent as in the group case. (Indeed, if $P=\bN \backslash \{1\}$, then $L \otimes I$ and $L\otimes L$ are not unitarily equivalent.) Nevertheless, in the case where $P$ is a right LCM submonoid of a group, Fell's absorption principle generalizes in the following form, which is well-known to experts.

\begin{proposition}\label{P:P coa B}
Let $P$ be a right LCM submonoid of a group $G$ and let $V = \{V_p\}_{p\in P}$ be a Nica-covariant representation of $P$.
Then the map
\begin{equation} \label{eq: LCMmap}
 L_p \longmapsto L_p \otimes V_p
\end{equation}
extends to a $*$-isomorphism between the ambient $\ca$-algebras.
\end{proposition}
\noindent In particular, one can take $V_p=L_p$, $p \in P$, in Proposition~\ref{P:P coa B} and therefore obtain that the left regular representation of a right LCM submonoid is self-absorbing, in the sense that (\ref{eq: LCMmap}) extends to a $*$-isomorphism. (We will not be defining right LCM submonoids or Nica representations in this paper as we do not make use of them; the reader should consult \cite{DK20, DKKLL} for the pertinent definitions and an extension of Proposition~\ref{P:P coa B} to product system.)

Perhaps counterintuitively, Proposition~\ref{P:P coa B} may not hold if one drops the assumption of $P$ being right LCM submonoid\footnote{ The author was not aware of this pathology and is grateful to Marcelo Laca and Camila Sehnem for making him aware of it.}. Indeed, it follows from \cite[Proposition 2.24]{Li12} that if $P$ does not satisfy independence, then the canonical map 
\[
L_p\longmapsto L_p\otimes L_p \in B(\ell^2(P)\otimes\ell^2(P))
\]
does not extend to a $*$-homomorphism of the generated $\ca$-algebras. Thus Fell's absorption principle fails in its most basic form as the left regular representation may not be self absorbing in general. The research in this paper was stimulated by this phenomenon and motivated us to search for possible generalizations of Fell's absorption principles in the spirit of Proposition~\ref{P:P coa B}. As it turns out, this can be done and the key step for such a generalization is an ``enhancement" of the left regular representation (Definition~\ref{def;enhanced}). In Corollary~\ref{C:firststep} we show that the enhanced left regular representation is indeed self-absorbing. Using this result as a first step, we introduce the $\ca$-algebra $\bTu(P)$, which is the full $\ca$-algebra of the Fell bundle associated with the enhanced left regular representation of $P$. We establish two versions of Fell's absorption principle for submonoids of groups. Theorem~\ref{thm;main1} generalizes Proposition~\ref{P:P coa B} to arbitrary submonoids of groups using the representations of $\bTu(P)$ as a substitute for Nica covariant representations. Theorem~\ref{thm;main2} characterizes which isometric representations of a semigroup $P$ are actually absorbed by its enhanced left regular representation.
In order to establish Theorem~\ref{thm;main2}, we discover that Li's semigroup $\ca$-algebra $\ca_s(P)$ \cite{Li12} and its representations play a pivotal role. Indeed, in Theorem~\ref{thm;Li} we show that the universal algebra $\bTu(P)$ is canonically isomorphic with Li's semigroup $\ca$-algebra $\ca_s(P)$. This result is not only a key ingredient in the proof of  Theorem~\ref{thm;main2} but it also leads to short proofs of earlier results that had more involved proofs. These improvements together with a pertinent discussion appear at the second half of Section~\ref{sec;Li} and at the end of Section~\ref{Sec;nsa}.

In Section~\ref{Sec;nsa} we look at the non-selfadjoint aspects of the theory. In Theorem~\ref{prop;identify} we show that the non-selfadjoint object attached to the enhanced left regular representation coincides with that attached to the left regular representation, i.e., the tensor algebra $\T_{\la}(P)^+$ of $P$. A significant generalization of this fact is later obtained in Theorem~\ref{P:P cover}. In Theorem~\ref{thm;char} we completely characterize the submonoids $P$ for which $\T_{\la}(P)^+$ admits a character $\omega$ satisfying $\om(L_p)=1$, for all $p \in P$.
In Theorem~\ref{thm;main3} we obtain a non-selfadjoint version of Fell's absorption principle. In particular, we obtain a characterization of which representations of $P$ extend to coactions of $P$ on the tensor algebra $\T_{\la}(P)^+$, provided that $P$ is left reversible and contained in an amenable group. 
We then use our non-selfadjoint version of Fell's absorption principle in order to improve recent results from Clouatre and Dor-On \cite{ClD} regarding the residual finite dimensionality of the maximal $\ca$-algebra of various tensor algebras. As another application, we give yet another proof for the existence of a $\ca$-algebra which is co-universal for equivariant, Li-covariant representations of a submonoid $P$ of a group $G$. This proof is quite elementary and avoids the use of partial crossed products or strong covariance. 

The paper ends with a short Section containing examples and open problems.

\section{Coactions and their Fell bundles} \label{S;cosystem}

In this introductory section, which borrows heavily from \cite[Section 3]{DKKLL},  we list several prerequisites regarding coactions and their associated Fell bundles. We will not be defining here what a Fell bundle is and we will not be listing the basic results of their $\ca$-algebraic theory. We suggest that the reader consults \cite{Exe97} for an expedited introduction and \cite{Exe17} for a detailed treatment. 

In what follows $G$ will always denote a fixed discrete group.
We write $u_g$, $g \in G$, for the generators of the universal group C*-algebras $\ca(G)$ and $l_g$, $g \in G$, for the generators of the reduced group $\ca$-algebra $\ca_r(G)$.
We write $l \colon \ca(G) \to \ca_r(G)$ for the canonical $*$-epimorphism.
Recall that $\ca(G)$ admits a faithful $*$-homomorphism
\[
\De \colon \ca(G) \to \ca(G) \otimes \ca(G) 
\text{ such that }
\De(u_g) = u_g \otimes u_g.
\]

\begin{definition}\label{D:cis coa}
Let $A$ be an operator algebra.
A \emph{coaction of a discrete group $G$ on $A$} is a completely isometric representation $\de \colon A \to A \otimes \ca(G)$ such that $\sum_{g \in G} A_g$ is norm-dense in $A$  for spectral subspaces 
\[
A_g := \{a \in A \mid \de(a) = a \otimes u_g\}. 
\]
If, in addition, the map $(\id \otimes l) \de$ is injective then the coaction $\de$ is called \emph{normal}. 

A map $\de$ as in Definition \ref{D:cis coa} automatically satisfies the coaction identity
\begin{equation} \label{eq;coactionid}
(\de \otimes \id_{\ca(G)}) \de = (\id_{A} \otimes \De) \de.
\end{equation}
Indeed, (\ref{eq;coactionid}) is readily seen to hold on $A_g$, $g \in G$, and therefore on $A$, since $\sum_{g \in G} A_g$ is norm-dense in $A$.
\end{definition}

It follows from the definition that if $\de \colon A \to A \otimes \ca(G)$  is a coaction then
\[
A_g \cdot A_h \subseteq A_{g h} \foral g, h \in G,
\]
because $\de$ is a homomorphism.
Conversely, if there are subspaces $\{A_g\}_{g \in G}$ such that $\sum_{g \in G} A_g$ is norm-dense in $A$ and a representation $\de \colon A \to A \otimes \ca(G)$ such that
\[
\de(a_g) = a_g \otimes u_g \foral a_g \in A_g, g \in G,
\]
then $\de$ is a coaction of $G$ on $A$.
Indeed $\de$ satisfies the coaction identity and it is completely isometric since $(\id_{A} \otimes \chi) \de = \id_{A}$.

In \cite[Remark 3.2]{DKKLL} it was noted that if $\de \colon A \to A \otimes \ca(G)$  is a coaction so that $\de$ extends to a $*$-homomorphism $\de \colon \ca(A) \to \ca(A) \otimes \ca(G)$ satisfying the coaction identity
\[
(\de \otimes \id) \de(c) = (\id \otimes \De) \de(c) \foral c \in \ca(A), 
\]
then $\de$ is automatically non-degenerate on $\ca(A)$, i.e.,
\[
\ol{\de(\ca(A)) \left[\ca(A) \otimes \ca(G)\right]} = \ca(A) \otimes \ca(G).
\]

In \cite[Remark 3.3]{DKKLL} it was noted that Definition \ref{D:cis coa} coincides with that of Quigg \cite{Qui96} when $A$ is a C*-algebra. Furthermore in this case $\de$ is a faithful $*$-homomorphism and we have that
\[
(A_g)^* = \{a^* \in A \mid \de(a^*) = a^* \otimes u_{g^{-1}} \} = A_{g^{-1}}.
\]
Also the coaction is then non-degenerate, i.e., it is a \emph{full} coaction.

The following gives a sufficient condition for the existence a compatible  normal  coaction.
\begin{proposition} \cite[Proposition 3.4]{DKKLL} \label{prop;KKLLco}
Let $A$ be an operator algebra and let $G$ be a group.
Suppose there are subspaces $\{A_g\}_{g \in G}$ such that $\sum_{g \in G} A_g$ is norm-dense in $A$, and there is a completely isometric homomorphism
\[
\de_\la \colon A \longrightarrow A \otimes \ca_r(G)
\]
such that
\begin{equation}\label{E:reduced}
\de_\la(a) = a \otimes l_g \foral a \in A_g, g \in G.
\end{equation}
Then $A$ admits a normal coaction $\de$ of $G$ satisfying  $\de_\la = (\id \otimes l) \de$.
\end{proposition}

Let us close this section with a discussion on gradings of C*-algebras in the sense of \cite{Exe17}.
\begin{definition}
Let $A$ be a C*-algebra and $G$ a discrete group. A collection of closed linear subspaces $\{\B_g\}_{g\in G}$ of $A$ is called a \emph{grading} of $A$ by $G$ if
\begin{enumerate}

\item $\B_g \B_h \subseteq \B_{gh}$

\item $\B_g^* = \B_{g^{-1}}$

\item $\sum_{g\in G} \B_g$ is dense in $A$.

\end{enumerate}
If in addition there is a conditional expectation $E : B \rightarrow \B_e$ which vanishes on $\B_g$ for $g\neq e$, we say that the pair $(\{\B_g\}_{g \in G}, E)$ is a \emph{topological} grading of $A$.
\end{definition}

When $\delta$ is a coaction on a C*-algebra $A$, the spectral subspaces $A_g$ for $g\in G$ comprise a topological grading for $A$ with conditional expectation $E_e = (\id \otimes F_e) \circ \de$ where $F_e : \ca(G) \rightarrow A$ is the $e$-th Fourier coefficient. Completely contractive maps $E_g : B \rightarrow B_g$ can be similarly defined by setting $E_g:= (\id \otimes F_g) \circ \de$, where $F_g : C^*(G) \rightarrow \bC $ is the $g$-th Fourier coefficient.

A grading of a C*-algebra by a group constitutes a Fell bundle over the group, and every Fell bundle arises this way, but not uniquely. Indeed, there may be many non-isomorphic graded C*-algebras whose gradings are all equal to a pre-assigned Fell bundle $\B$. At one extreme sits the maximal C*-algebra $C^*(\B)$, which is universal  for representations of $\B$, while at the other extreme is the minimal (reduced) cross sectional algebra $C^*_r(\B)$ which is defined via the left regular representation  of $\B$.  We refer again to \cite{Exe97, Exe17} for the precise definitions and details.

\section{Fell's absorption principle for semigroup $\ca$-algebras}

If $\H$ is a Hilbert space then $\H^{\otimes n }$ will denote its $n$-fold tensor product. Similarly, if $S$ is a bounded operator on $\H$ (or a representation of a semigroup on $\H$) then $S^{\otimes n}$ will denote the $n$-fold tensor product.

\begin{definition} \label{def;enhanced}
Let $P$ be a semigroup. The \textit{enhanced left regular representation} of $P$ is the representation 
\[
\eL: P\longrightarrow B\left(\mathlarger{\mathlarger{\oplus}}_{n=1}^{\infty} \ell^2(P)^{\otimes n } \right); \,\, P\ni p \longmapsto \loplus_{n=1}^{\infty}  L_p^{\otimes n},
\]
where $ L = \{L_p\}_{p \in P}$ denotes the left regular representation of $P$. (For notational simplicity, we will be writing $\eel$ instead of  $\oplus_{n=1}^{\infty} \ell^2(P)^{\otimes n }$.)
\end{definition}
 
In order to study the enhanced left regular representation in the context of $\ca$-algebras, we need to built a universal $\ca$-algebra with the enhanced left regular representation as a distinguished representation. For that reason, we will show that the algebra $\bTl(P):=\ca(\eL)$, being a reduced type object, admits a normal coaction and then we use Fell bundle theory.  But first we need to establish some notation. 

Let $P$ be a semigroup of a group $G$. For each $ k \in \bN$ we consider the set of words of length $2k$ in $P$,
\[
\W(P)^k:= \{(p_1, p_2, \dots , p_{2k-1}, p_{2k})\mid p_j \in P, \mbox{ for } j=1, 2, \dots, 2k\}
\]
and we let $\W(P):=\cup_{k =0}^{\infty}\W(P)^k$, with the understanding $W(P)^0 = \emptyset$. (When the context makes it clear what $P$ is, we simply write $\W$ instead of $\W(P)$.) With each word $a \in \W^k$ we make the assignment 
\[
a=(p_1, p_2, \dots , p_{2k-1}, p_{2k}) \longmapsto  \dot{a} :=p_1^{-1}p_2, \dots p_{2k-1}^{-1} p_{2k} \in G.
\]
A word $a \in \W$ is said to be neutral if $\dot{a}=e$, the neutral element of $G$. If $V = \{V_p\}_{p\in P}$ is an isometric representation of $P$ and $a=(p_1, p_2, \dots , p_{2k-1}, p_{2k})  \in \W$, then we define
\[
\dot{V}_{a} := V_{p_1}^*V_{p_2}V_{p_3}^*  \cdots V_{p_{2k-1}}^*V_{p_{2k}}.
\]
\begin{proposition}\label{P:f coaction}
Let $P$ be a submonoid of a group $G$. Then there is a normal coaction
\[
\bar{\de}\colon \bTl(P) \longrightarrow \bTl (P)  \otimes \ca(G) ; \eL_p \longmapsto \eL_p \otimes u_{p}.
\]
Moreover each spectral subspace $ \bTl(P)_g$, $g \in G$, of $\bar{\de}$ satisfies
\[
 \bTl(P)_g = \cspan\{ \dot{\eL}_a \mid a \in \W, \dot{a}=g\}.
\]
\end{proposition}

\begin{proof}
It is enough to establish a normal coaction
\[
\de \colon \T_{\la}(P) \longrightarrow \T_\la(P)  \otimes \ca(G) : L_p \longmapsto L_p \otimes u_{p}
\]
so that each spectral space $ \T_\la(P)_g$, $g \in G$,  satisfies
\[
 \T(P)_g = \cspan\{ \dot{L}_a \mid a \in \W, \dot{a}=g\}.
\]
Then the desired coaction is the restriction of
\[
\begin{split}
\loplus_{n=1}^{\infty} (\id^{\otimes n} \otimes \de ) : \loplus_{n=1}^{\infty} \T_{\lambda}(P)^{\otimes n} \longrightarrow 
&\loplus_{n=1}^{\infty} (\T_{\lambda}(P)^{\otimes n} \otimes \ca(G)) \\
& \,\, \simeq
\left( \loplus_{n=1}^{\infty} \T_{\lambda}(P)^{\otimes n} \right) \otimes \ca(G)) 
\end{split}
\]
on $ \hat{\T}_\la(P)$.

The existence of such a coaction $\delta$ on $\T_{\lambda}(P)$ is well-known. 
Indeed, consider the operator $U \colon \ell^2(P) \otimes \ell^2(G) \to  \ell^2(P)  \otimes \ell^2(G)$ be given by
\[
U (\de_p \otimes \de_g) = \de_p \otimes \de_{p g}
\foral
p \in P, g \in G.
\]
Then $U$ is a unitary in $B( \ell^2(P) \otimes \ell^2(G) )$ satisfying $U(L_p \otimes I) = (L_p \otimes l_p) U$ for all $p \in P$.
Thus $\ad_{U}$ implements a normal coaction that promotes to desired coaction $\bar{\de}$ by Proposition~\ref{prop;KKLLco}.
\end{proof}

The existence of the coaction $\de$ implies that other natural algebras associated with the representation $\eL$ also admit coactions whose spectral subspaces have similar descriptions, e.g., 
\[
\ca(\eL\otimes \eL) \subseteq   \hat{\T}_\la(P) \otimes  \hat{\T}_\la(P).
\]
This will be a recurring theme in this paper. We now need the following

\begin{lemma}[cf. Lemma 1.4 of \cite{LR96}] \label{basicLem}
Suppose $F$ is a finite family of mutually commuting projections in a unital $\ca$-algebra and let $\la_X \in \bC$ for each $X \in F$. For every subset $A$ of $F$ define
\[
Q_A:=\prod_{X \in A} X \prod_{X \in F\backslash A}(1-X),
\]
which includes the cases $Q_{\emptyset}=\prod_F(1-X)$ and $Q_F= \prod_FX$.

Then $1 = \sum_{A \subseteq F}Q_A$ is a decomposition of the identity in mutually orthogonal projections,
\[
\sum_{X \in F}\la_X X=\sum_{\emptyset \neq A\subseteq F}\left(\sum_{X\in A} \la_X\right)Q_A
\]
and
\[
\Big\| \sum_{X \in F}\la_X X \Big\| = \max\Big\{ \big| \sum_{X \in F}\la_X \big| \mid \emptyset \neq A\subseteq F , \,\, Q_A\neq 0\Big\}.
\]
\end{lemma}

\begin{proposition} \label{doubling}
Let $P$ be a submonoid of a group $G$. The mapping 
\begin{equation} \label{n=2}
\eL_p=\loplus_{n=1}^{\infty} L_p^{\otimes n} \longmapsto \eL_p\otimes L_p = \loplus_{n=2}^{\infty} L_p^{\otimes n}, \,\, p \in P,
\end{equation}
extends to an isometric $*$-representation of $\bTl(P)$.
\end{proposition}

\begin{proof}
It is elementary to verify that (\ref{n=2}) extends to a well-defined $*$-representation $\rho$ of $\bTl$. Now arguments similar to that of Lemma~\ref{P:f coaction} show that $\ca(\rho)$ admits a coaction $\overline{\de}_{\rho}$, whose spectral subspaces $\ca(\rho)_g$, $g \in G$, have a similar description to that of $\overline{\de}$. Hence both $\bTl(P)$ and $\ca(\rho)$ become graded $\ca$-algebras and $\rho$ induces a surjective bundle homomorphism between the associated Fell bundles. 

We show now that $\rho$ is isometric on $\bTl(P)_e$.
Consider a non-zero element of $\bTl(P)$ of the form 
\begin{equation} \label{specialsum}
\alpha := \sum_F \la_{a}\dot{\eL}_{a},
\end{equation}
where $F$ is a finite collection of neutral words. We claim that $\|\rho (\alpha) \|= \|\alpha\|$. By way of contradiction assume otherwise. Then the second half of Lemma~\ref{basicLem} implies the existence of a non-empty $A \subseteq F$ so that the projection
\[
Q_A:=\prod_{a \in A} \dot{\eL}_a\prod_{b \in F\backslash A}(1 -\dot{\eL}_b)
\]
satisfies $\|Q_A\alpha\|=\|\alpha\| $ and yet $\rho(Q_A) = 0$. However for the non-zero projection $Q_A$ to satisfy $\rho(Q_A) = 0$, we must have that
\[
\prod_{a \in A} \dot{L}_a\prod_{b \in F\backslash A}(1 -\dot{L}_b) \neq 0
\]
while all other direct summands satisfy
\[
\prod_{a \in A} \dot{L}^{\otimes n}_a\prod_{b \in F\backslash A}(1 -\dot{L}^{\otimes n}_b) = 0, \,\, n=2, 3, \dots.
\]
However 
\[
(1-\dot{L}_b )^{\otimes 2} \leq 1 - \dot{L}_b^{\otimes 2} \,\, \mbox{ for all } b \in F\backslash A,
\]
and so 
\[
\Big( \prod_{b \in F\backslash A} (1-\dot{L}_b )\Big)^{\otimes 2} \leq \prod_{b \in F\backslash A} (1 - \dot{L}_b^{\otimes 2}).
\]
But then,
\begin{equation}
\begin{split}
0\neq \Big( \prod_{a \in A} \dot{L}_a\prod_{b \in F\backslash A} (1-\dot{L}_b )\Big)^{\otimes 2} &=
 \Big(\prod_{a \in A} \dot{L}_a\Big)^{\otimes 2} \Big(\prod_{b \in F\backslash A} (1-\dot{L}_b )\Big)^{\otimes 2} \\
 &\leq \prod_{a \in A} \dot{L}^{\otimes 2}_a\prod_{b \in F\backslash A}(1 -\dot{L}^{\otimes 2}_b) \\
 &\leq \rho(Q_A).
\end{split}
\end{equation}
This contradiction shows that $\rho$ is isometric on sums of the form (\ref{specialsum}) and therefore on the generated subspace of $\bTl(P)$, i.e., $\bTl(P)_e$.

From the above, it follows that $\rho$ induces a bundle isomorphism between $\{ \bTl(P)_g\}_{g \in G}$ and $\{\ca(\rho) \}_{g \in G}$ and so $\rho$ induces a $*$-isomorphism between the reduced cross sectional $\ca$-algebras $\ca_{r}(\{ \bTl(P)_g\}_{g \in G})$ and $\ca_r(\{\ca(\rho)_g \}_{g \in G})$ which coincides with $\rho$ on fibers. However, both $\bTl(P)$ and $\ca(\rho)$ admit \textit{faithful} expectations that act as the identity on the unit fibers and annihilate all other fibers. Hence by \cite[Proposition 3.7]{Exe97} we have 
\[
\ca_{r} (\{\bTl(P)_g\}_{g \in G}) \simeq \bTl(P) \mbox{ and } \ca_r(\{\ca(\rho)_g \}_{g \in G}) \simeq \ca(\rho)
\]
canonically and the proof of the proposition is complete.
\end{proof}

Unlike the left regular representation, the enhanced left regular representation of a semigroup is always self-absorbing. Indeed

\begin{corollary} \label{C:firststep}
Let $P$ be a submonoid of a group $G$. The mapping 
\begin{equation*}
\bTl(P) \ni \eL_p \longmapsto \eL_p  \otimes \eL_p \in \ca(\eL\otimes \eL)
\end{equation*}
extends to a $*$-isomorphism between the ambient $\ca$-algebras.
\end{corollary}

\begin{proof}
It is easy to see that the mapping
\[
\eL_p\otimes \eL_p \longmapsto \loplus_{n=2}^{\infty} L_p^{\otimes n}, \,\, p \in P,
\]
extends to an isometric representation of $\ca(\eL\otimes \eL)$. In the previous Proposition we established that 
\[
\eL_p=\loplus_{n=1}^{\infty} L_p^{\otimes n} \longmapsto \loplus_{n=2}^{\infty} L_p^{\otimes n}, \,\, p \in P,
\]
extends to an isometric $*$-representation of $\bTl(P)$ and the conclusion follows.
\end{proof}

In order to obtain our isometric Fell's absorption principle, we now need to identify a suitable analogue of Nica-covariance that applies to isometric representations of arbitrary semigroups.

It is known \cite[Proposition 4.3]{DKKLL} that the Nica-covariant representations of a right LCM semigroup $P$ are exactly the isometric representations of $P$ which extend to $*$-representations of the full cross-sectional $\ca$-algebra of the Fell bundle $\{ \T_{\la}(P)_g\}_{g \in G}$ arising from the coaction $\de$ appearing in the proof of Proposition~\ref{P:f coaction}. Following this lead, let $$\bTu (P):= \ca(\{\bTl(P)_g\}_{g \in G})$$ be the full cross sectional $\ca$-algebra of the Fell bundle $\{\bTl(P)_g\}_{g \in G}$. Let $\bt$ be the natural $*$-embedding of $\{\bTl(P)_g\}_{g \in G}$ in $\bTu (P)$ and let $\bl: \bTu(P) \rightarrow \bTl(P)$ be the $*$-homomorphism that makes the following diagram 
\[
\xymatrix{
& & \bTu(P)\ar@{>}[d]^{\bl} \\
\{\bTl(P)_g\}_{g \in G} \ar[rru]^{\bt} \ar@{^{(}->}[rr]& & \bTl(P)}
\]
commutative. For ease of notation, we write $\bar{t}_p$ instead of $\bar{t}(\eL_p)$, $ p \in P$. Observe that since $\bTu(P)$ and $\bTl (P)$ are cross sectional $\ca$-algebra of the same bundle, $\bl: \bTu(P) \rightarrow \bTl(P)$ is a $*$-isomorphism on the unit fiber $\bTu (P)_e$.

\begin{lemma} \label{lem;req2}
Let $P$ be a submonoid of a group $G$. Then the map
\begin{equation} \label{eq;step1}
\ca(\eL\otimes \eL) \ni \eL_p\otimes \eL_p \longmapsto \eL_p\otimes \bar{t}_p \in \ca(\eL\otimes \bt ), \, \,\, p \in P,
\end{equation}
extends to a  $*$-isomorphism between the generated $\ca$-algebras.
\end{lemma}
\begin{proof}
Arguing as in the proof of Proposition \ref{P:f coaction}, one can see that the $\ca$-algebras $\ca(\eL\otimes \eL) $ and $\ca(\eL\otimes \bt)$ admit coactions whose spectral subspaces $\{ \ca(\eL\otimes \eL)_g\}_{g \in G} $ and $\ca(\eL\otimes \bt)_g \}_{g \in G}$ respectively satisfy
\[
\ca(\eL\otimes \eL)_g =  \cspan\{ \dot{\eL}_a \otimes  \dot{\eL}_a \mid a \in \W, \dot{a}=g\}
\]
and 
\[
\ca(\eL\otimes \bt)_g =  \cspan\{ \dot{\eL}_a   \otimes \dot{\bt}_a \mid a \in \W, \dot{a}=g\},
\]
for all $g \in G$. Furthermore both $\ca(\eL\otimes \eL) $ and $\ca(\eL\otimes \bt)$ admit faithful expectations that act as the identity map on the zero fibers $\ca(\eL\otimes \eL)_e$ and $ \ca(\eL\otimes \eL)_e$ and annihilate all other spectral subspaces. Hence 
\[
\ca(\eL\otimes \eL) \simeq\ca_r\left(\{\ca(\eL\otimes \eL)_g\}_{g \in G}\right) \mbox{ and } 
\ca(\eL\otimes \bt) \simeq\ca_r\left(\{\ca(\eL\otimes \bt)_g\}_{g \in G}\right) \]
Now the map
\[
\id \otimes \bl: \bTl(P)\otimes \bTu(P) \longrightarrow \bTl(P)\otimes \bTl(P)
\]
establishes a bundle homomorphism from the Fell bundle $\{ \ca(\eL\otimes \bt)_g \}_{g \in G}$ onto $\{ \ca(\eL\otimes \eL)_g\}_{g \in G} $. However the restriction of $\id \otimes \bl$ on $ \bTl(P)_e\otimes \bTu(P)_e$ is injective because ${\bl}_{\mid_{ \bTu(P)_e}} $ is injective. Since 
\[
\ca(\eL\otimes \bt)_e \subseteq \bTl(P)\otimes \bTu(P) 
\]
the bundle homomorphism from the Fell bundle $\{ \ca(\eL\otimes \bt)_g \}_{g \in G}$ onto $\{ \ca(\eL\otimes \eL)_g\}_{g \in G}$ is actually an isomorphism, which promotes to a $*$-isomorphism between their reduced cross-sectional $\ca$-algebras. This completes the proof.
\end{proof}
The Nica-covariant representations of a right LCM semigroup are \textit{faithful on monomials}, i.e., if $\{V_p\}_{p \in P}$ is a Nica-covariant representation of a right LCM semigroup $P$ and 
\[
L_{p_1}^*L_{p_2}L_{p_3}^*  \cdots L_{p_{2k-1}}^*L_{p_{2k}} \neq 0 \mbox{ for some } p_1, p_2, \dots p_{2k} \in P,
\]
then 
\[
V_{p_1}^*V_{p_2}V_{p_3}^*  \cdots V_{p_{2k-1}}^*V_{p_{2k}} \neq 0.
\]
This property follows directly from the definition of a right LCM group but one can also deduce it from Proposition~\ref{P:P coa B}. Indeed, if $\{V_p\}_{p \in P}$ is a Nica-covariant representation of a right LCM semigroup $P$, then the injectivity of the $*$-homomorphism coming from (\ref{eq: LCMmap}) implies that $\{V_p\}_{p \in P}$ is necessarily faithful on monomials because the left regular representation is. We now observe that this property also holds for our analogue of Nica-covariant representations, which appear in (\ref{eq;Lidef}) below.

\begin{lemma} \label{lem;faithmon}
Let $P$ be a submonoid of a group $G$ and let $\pi: \bTu (P)\rightarrow B(\H)$ be a $*$-representation. Then the semigroup representation 
\begin{equation} \label{eq;Lidef}
P\ni p \longmapsto \pi (\bt_p) \in B(\H)
\end{equation}
is faithful on monomials.
\end{lemma}

\begin{proof}
We need to show that if $a\in W$ is a neutral word with $\dot{\eL}_a\neq 0$, then $\pi (\dot{\bt}_a) \neq 0$.
Let $\{\de_p\}_{p \in P}$ be the canonical basis of $\ell^2(P)$. Since $\dot{\eL}_a\neq 0$, we have that $\dot{L}_a\neq 0$ and so there exists $q \in P$ with $\de_q \in \dot{L}_a(\ell^2(P))$. Therefore
\[
\dot{L}_q\dot{L}_q^*(\ell^2(P)) = \cspan\{\de_{qp} \mid p \in P\} \subseteq \dot{L}_a(\ell^2(P)) 
\]
and so $ \dot{L}_a\dot{L}_q\dot{L}_q^*= \dot{L}_q\dot{L}_q^*$. This in turn implies $ \dot{\eL}_a\dot{\eL}_q\dot{\eL}_q^*= \dot{\eL}_q\dot{\eL}_q^*$ and since $\bl: \bTu(P) \rightarrow \bTl(P)$ is a $*$-isomorphism on the unit fiber $\bTu (P)_e$, we have $\dot{\bt}_a \dot{\bt}_q  \dot{\bt}_q^* = \dot{\bt}_q  \dot{\bt}_q^*$.
Apply now $\pi$ to obtain 
\begin{equation} \label{eq;faithmon}
\pi(\dot{\bt}_a) \pi( \dot{\bt}_q) \pi( \dot{\bt}_q^*) = \pi( \dot{\bt}_q) \pi( \dot{\bt}_q^*) .
\end{equation}
However $\pi( \dot{\bt}_q)$ is an isometry and so $\pi( \dot{\bt}_q) \pi( \dot{\bt}_q^*) \neq0$. This implies that the right side of (\ref{eq;faithmon}) and in particular $\pi(\dot{\bt}_a)$ is non zero, as desired.
\end{proof}

We have arrived to our first version of Fell's absorption principle.

\begin{theorem}[Fell's absorption principle for semigroups; first version] \label{thm;main1}
Let $P$ be a submonoid of a group $G$. Let $\eL$ be the enhanced left regular representation of $P$ and let $\pi$ be a $*$-representation of $\bTu (P)$.
Then the map
\begin{equation} \label{eq;Fell1}
\bTl (P)\ni  \eL_p \longmapsto \eL_p \otimes \pi(\bt_p),  \,\, p \in P,
\end{equation}
extends to an injective representation of $\bTl (P)$.
\end{theorem}

\begin{proof}
By combining Corollary~\ref{C:firststep} with Lemma~\ref{lem;req2}, we conclude that the map 
\[
\bTl \ni  \eL_p \longmapsto \eL_p \otimes \bt_p \in \ca(\eL\otimes \bt),  \,\, p \in P,
\]
extends to a $*$-homomorphism (actually isomorphism) between the ambient $\ca$-algebras. By composing this $*$-homomorphism with $$ \id\otimes \pi: \bTl (P) \otimes \bTu (P)\rightarrow  \bTl (P) \otimes \pi(\bTu (P) ), $$ we obtain that the map in (\ref{eq;Fell1}) extends to a (well-defined) $*$-homomorphism $\rho$ between the ambient $\ca$-algebras. It remains to establish its injectivity.

Arguments similar to that of Lemma~\ref{P:f coaction} show that $\ca(\rho)$ admits a coaction $\overline{\de}_{\rho}$, whose spectral subspaces $\ca(\rho)_g$ satisfy
\[
\ca(\rho)_g =  \cspan\{ \dot{\eL}_a   \otimes \pi(\dot{\bt}_a) \mid a \in \W, \dot{a}=g\}, \,\, g \in G.
\]
Hence both $\bTl(P)$ and $\ca(\rho)$ become graded $\ca$-algebras and $\rho$ induces a surjective bundle homomorphism between the associated Fell bundles.

We show now that $\rho$ is isometric on $\bTl(P)_e$.
Consider a non-zero element of $\bTl(P)$ of the form 
\begin{equation} \label{specialsum2}
\alpha := \sum_F \la_{a}\dot{\eL}_{a},
\end{equation}
where $F$ is a finite collection of neutral words. We claim that $\|\rho (\alpha) \|= \|\alpha\|$. The second half of Lemma~\ref{basicLem} implies the existence of a non-empty $A \subseteq F$ so that the projection
\[
Q_A:=\prod_{a \in A} \dot{\eL}_a\prod_{b \in F\backslash A}(1 -\dot{\eL}_b)
\]
satisfies $\|Q_A\alpha\|=\|\alpha\| $. It suffices to show that $\rho(Q_A) \neq 0$. 

Indeed, note that 
\[
(1-\dot{\eL}_b)\otimes 1 \leq 1-\dot{\eL}_b\otimes \pi(\dot{\bt}_b),
\]
for all $b \in F\backslash A$. Furthermore $\prod_{a \in A} \pi(\dot{\bt}_a) \neq 0$ because the representation $\pi$ is faithful on monomials (Lemma~\ref{lem;faithmon}). Therefore
\[
\begin{split}
0\neq Q_A\otimes (\prod_{a \in A} \pi(\dot{\bt}_a ))&= \Big(\prod_{a \in A} \dot{\eL}_a\prod_{b \in F\backslash A}(1 -\dot{\eL}_b)\Big) \otimes \Big(\prod_{a \in A} \pi(\dot{\bt}_a)  \Big) \\
     &= \Big( \prod_a\dot{\eL}_a  \otimes \pi(\dot{\bt}_a)  \Big) \prod_{b \in F\backslash A}(1 -\dot{\eL}_b)\otimes 1\\
     &\leq \Big( \prod_a\dot{\eL}_a  \otimes \pi(\dot{\bt}_a ) \Big) \prod_{b \in F\backslash A}(1-\dot{\eL}_b\otimes \pi(\dot{\bt}_b))\\
     &=\rho\Big( \prod_{a \in A} \dot{\eL}_a\prod_{b \in F\backslash A}(1 -\dot{\eL}_b)\Big) = \rho(Q_A).
\end{split}
\]
This arguments above show that $\rho$ is isometric on sums of the form (\ref{specialsum2}) and therefore on the generated subspace of $\bTl(P)$, i.e., $\bTl(P)_e$.

From the above, it follows that $\rho$ induces a bundle isomorphism between $\{ \bTl(P)_g\}_{g \in G}$ and $\{\ca(\rho) \}_{g \in G}$ and so $\rho$ induces a $*$-isomorphism between the reduced cross sectional $\ca$-algebras $\ca_{r}(\{ \bTl(P)_g\}_{g \in G})$ and $\ca_r(\{\ca(\rho)_g \}_{g \in G})$ which coincides with $\rho$ on fibers. However, both $\bTl(P)$ and $\ca(\rho)$ admit \textit{faithful} expectations that act as the identity on the unit fiber and annihilate all other fibers. Hence by \cite[Proposition 3.7]{Exe97} we have that
\[
\ca_{r} (\{\bTl(P)_g\}_{g \in G}) \simeq \bTl(P) \mbox{ and } \ca_r(\{\ca(\rho)_g \}_{g \in G}) \simeq \ca(\rho)
\]
canonically and the proof of the theorem is complete.
\end{proof}

\begin{remark}
In the case where $P$ is a right LCM submonoid of a group $G$, Proposition~\ref{P:P coa B} implies that the map
\[
\T_{\la}(P) \ni L_p\longmapsto \eL_p \in \bTl(P)
\]extends to an isomorphism between the ambient $\ca$-algebras and so Theorem~\ref{thm;main1} reduces to  Proposition~\ref{P:P coa B} in that case.
\end{remark}

\section{$\bTu(P)$ coincides with Li's semigroup $\ca$-algebra $\ca_s(P)$} \label{sec;Li}

Theorem~\ref{thm;main1} brings on the spotlight the representations of $P$ coming from $\bTu(P)$. It is therefore desirable to have an explicit description for such representations in terms of relations, which will make them more usable. Such a description can be achieved by introducing Li's semigroup $\ca$-algebra $\ca_s(P)$ \cite{Li13} and its defining relations. This will allow us to give a more concrete version of Fell's absorption principle and actually characterize which isometric representations of $P$ can be absorbed by the enhanced left regular representation. On the other hand, we will be able to connect Li's semigroup $\ca$-algebra with a very concrete object, i.e., $\bTl (P)$. Using this connection we will give easy to remember proofs of some earlier results requiring the consideration of inverse semigroup theory. More applications will appear in the next Section. For the moment, we need some additional notation.

If $a = (p_1, p_2, \dots p_{2k-1}, p_{2k})$ is a word in $\W(P)$, we write
\[
K(a) := P\cap(p^{-1}_{2k}p_{2k-1})P\cap (p^{-1}_{2k}p_{2k-1} p^{-1}_{2k-2}p_{2k-3})P\cap \dots \cap (\dot{\tilde{a}})P,
\]
for the \textit{constructible right ideal} associated with $a$. It is easy to see that if $\{ \delta_p\}_{p \in P}$ is the canonical orthonormal basis for $\ell^2(P)$, then 
\[
K(a)=\{p\mid \dot{L}_a\de_p=\de_p, p\in P\}.
\]
We let 
\[
\fJ(P):= \{K(a)\mid a \in \W(P)\},
\]
dropping the reference to $P$ and simply writing $\fJ$, if there is no source of confusion.

\begin{definition}[Definition 3.2 of \cite{Li13}]
Let $P$ be a submonoid of a group $G$ and let $\fJ$ be the collection of all constructible right ideals of $P$. Li's semigroup $\ca$-algebra of $P$, denoted as $\ca_s(P)$, is the universal $\ca$-algebra generated by a family of isometries $\{v_p\}_{p \in P}$ and projections $\{e_S\}_{S \in \fJ\cup\{\emptyset\}}$ such that
\begin{itemize}
\item[(i)] $v_p v_q =v_{pq}$ \,\, whenever $p,q\in P$;
\item[(ii)] $e_{\emptyset} =0$
\item[(iii)] $\dot{v}_a = e_S$ whenever $S \in \fJ$ and $a \in \W$ satisfy $\dot{a}=e$ and $K(a)=S$.
\end{itemize}
\end{definition}

It turns out that $\ca_s (P)$ admits a ``lighter" set of axioms that makes it easier to identify its representations. Indeed, Laca and Sehnem show in \cite[Proposition 3.22]{LaS} that $\ca_s(P)$ is canonically isomorphic to the universal $\ca$-algebra generated by a family of elements $\{w_p\mid p \in P\}$ subject to the relations
\begin{itemize}
\item[(T1)] $w_e = 1$;
\item[(T2)] $\dot{w}_a=0$, if $K(a)=\emptyset$ with $\dot{a}=e$;
\item[(T3)] $\dot{w}_a=\dot{w}_b$ if $a$ and $b$ are neutral words with $K(a)=K(b)$.
\end{itemize}
In particular, any map $w:P\rightarrow B(\H)$ satisfying the relations (T1), (T2) and (T3) is a representation of $P$ by isometries.  In this paper we will use the Laca and Sehnem picture for $\ca_s(P)$ and we will call any representation satisfying the relations (T1), (T2) and (T3) a \textit{Li-covariant representation} of $P$.

\begin{lemma} \label{L;trivial}
Let $X$ be a set and $\{X_i\}_{i=1}^{k}$ be subsets of $X$ so that their $n$-fold cartesian products satisfy  
\[
X^{(n)} = \cup_{i=1}^{k} X_i^{(n)} \mbox{  for all } n \in \bN. 
\]
Then there exists $i_0\in \{1, 2, \dots , k\}$ so that $X = X_{i_0}$.
\end{lemma}

\begin{proof}
Indeed if not, then for each $1\leq i \leq k$, choose $x_i\in X\backslash X_i$ and note that 
\[
(x_1, x_2, \dots x_k) \in X^{(k)}\backslash \cup_{i=1}^k X^{(k)}_i , 
\]
a contradiction.
\end{proof}

\begin{theorem} \label{thm;Li}
Let $P$ be a submonoid of a group $G$. Then $\bTu(P)$ is canonically isomorphic with Li's semigroup $\ca$-algebra $\ca_s (P)$.
\end{theorem}

\begin{proof}
Since all three  relations (T1), (T2) and (T3) manifest in the unit fiber of $\bTl (P)$ with $w_p=\eL_p$, $p \in P$, the isometries $\{\bt_p \}_{p\in P} \subseteq \bTu (P)$ will also satisfy properties (T1), (T2) and (T3). By universality, there exists a $*$-homomorphism 
\[
\rho :  \ca_s (P) \longrightarrow \bTu (P); w_p \longmapsto \bt_p, \,\, \mbox{ for all }p \in P.
\]
In order to produce an inverse for $\rho$, we need to establish that $\ca_s (P)$ is a cross-sectional $\ca$-algebra for the Fell bundle $\{\bTl (P)_{g}\}_{g \in G}$. 

It is easily seen that $\ca_s (P)$ admits a coaction $\de_s: \ca_s (P) \rightarrow \ca_s (P) \otimes \ca(G)$ with spectral subspaces satisfying 
\[
 \ca_s(P)_g = \cspan\{ \dot{w}_a \mid a \in \W, \dot{a}=g\}, \,\, g \in G.
\]
Let $\rho_g$ be the restriction of $(\bl  )\circ \rho$ on the fiber $\ca_s (P)_g$, $g \in G$.  We will show that $\{\rho_g\}_{g \in G}$ promotes to an isomorphism between the Fell bundles $\{  \ca_s(P)_g  \}_{g \in G}$ and $\{\bTl (P)_{g}\}_{g \in G}$.

Consider a non-zero element of $\ca_s(P)$ of the form 
\begin{equation} \label{specialsum}
\alpha := \sum_F \la_{a}\dot{w}_{a},
\end{equation}
where $F$ is a finite collection of neutral words. 

\vspace{.1in}

\noindent \textbf{Claim:} $\| \rho_e(\al )  \|= \| \al \|$.

\vspace{.05in}

\noindent \textit{Proof of Claim.} By way of contradiction assume otherwise. Then the second half of Lemma~\ref{basicLem} implies the existence of a non-empty $A \subseteq F$ so that the projection
\[
Q_A:=\prod_{a \in A} \dot{w}_a\prod_{b \in F\backslash A}(1 -\dot{w}_b)
\]
satisfies $\|Q_A\alpha\|=\|\alpha\| $ and yet
\[
\rho_e(Q_A) =\prod_{a \in A} \dot{\eL}_a\prod_{b \in F\backslash A}(1 -\dot{\eL}_b) =0.
\]
By concatenating all words in $A$ we obtain $a_0\in \W$ so that 
\[
\dot{w}_{a_0} = \prod_{a \in A} \dot{w}_a \,\, \mbox{ and } \,\,  \dot{\eL}_{a_0} = \prod_{a \in A} \dot{\eL}_a 
\]
and so 
\begin{equation} \label{eq;claim}
Q_A=\prod_{b \in F\backslash A}(\dot{w}_{a_0} -\dot{w}_{a_0 b})
\end{equation}
while 
\[
\rho_e(Q_A )=\prod_{b \in F\backslash A}(\dot{\eL}_{a_0} -\dot{\eL}_{a_0 b})  
= \loplus_{n=1}^{\infty} \big( \prod_{b \in F\backslash A}(\dot{L}_{a_0}^{\otimes n} -\dot{L}_{a_0 b}^{\otimes n}) \Big)    =0.
\]
The right side of the last equality implies that $K(a_0)^{(n)} = \cup_{b \in F\backslash A} K(a_0 b)^{(n)}$, for all $n \in \bN$, and so Lemma~\ref{L;trivial} implies that there exist $b_0 \in F\backslash A$ so that $K(a_0)=K(a_0 b_0)$. Hence (T3) implies that $\dot{w}_{a_0}  = \dot{w}_{a_0 b_0}$ and so by substituting in (\ref{eq;claim}) we obtain that $Q_A=0$, a contradiction that proves the claim.

\vspace{.05in}

The claim shows now that $\rho_e$ is isometric on a dense subset of $\ca_s (P)_e$ and so an injective $*$-isomorphism between $\ca_s (P)_e$ and $\bTl (P)_e$. Using the $\ca$-identity, it follows that $\rho_g$ is well defined and isometric for all $g \in G$. It is routine to verify the rest of the properties that establish that $\{\rho_g\}_{g \in G}$ is an isomorphism between the Fell bundles $\{ \ca_s(P)_g \}_{g \in G}$ and $\{\bTl (P)_{g}\}_{g \in G}$, as desired. 

Since $\ca_s (P)$ is a cross-sectional $\ca$-algebra for the Fell bundle $\{\bTl (P)_{g}\}_{g \in G}$, there exists a surjective $*$-homomorphism 
\[
\bTu (P)= \ca ( \{\bTl (P)_{g}\}_{g \in G}) \longrightarrow \ca_s (P); \bt_p \longmapsto w_p, \,\, \mbox{ for all } p \in P.
\]
This is the desired inverse for $\rho$ and the proof of the theorem is complete.
\end{proof}

\begin{remark} \label{rem;free}
It follows now from the proof of Theorem~\ref{thm;Li} that the unit fiber of $\{ \ca_s (P)_g \}_{g \in G}$ demonstrates a remarkable property which we record here for future reference. 

\textit{If $\{ a_i\}_{i=0}^n$ is a finite collection of neutral words in $\W$ and $\prod_{i=1}^n (\dot{w}_{a_0}-\dot{w}_{a_0 a_i })=0$, then necessarily $\dot{w}_{a_0 }= \dot{w}_{a_0a_i}$ for some $i=1, 2, \dots, n$.}
\end{remark}

In \cite[Introduction: Question 3]{KKLL} the authors ask: is there a co-universal algebra and what is its form? We can answer this question fully for $\ca_s(P)$ in the case where the submonoid $P$ embeds in an amenable group. 

\begin{corollary} \label{cor;question}
Let $P$ be a submonoid of an amenable group $G$. Then Li's semigroup $\ca$-algebra $\ca_s(P)$ is canonically isomorphic with $\bTl (P)$, i.e., the $\ca$-algebra generated by the enhanced left regular representation.
\end{corollary}

We now state Fell's absorption principle for submonoids of groups and characterize which isometric representations can actually be absorbed by the enhanced left regular representation.

\begin{theorem}[Fell's absorption principle for semigroups; second version] \label{thm;main2}
Let $P$ be a submonoid of a group $G$ and let $\eL$ be the enhanced left regular representation of $P$. If $V= \{ V_p\}_{p \in P}$ is an isometric representation of $P$, then the following are equivalent
\begin{itemize}
\item[\textup{(i)}] The map
\begin{equation*} \label{eq;Fell}
\bTl(P) \ni \eL_p \longmapsto \eL_p \otimes V_p,  \,\, p \in P,
\end{equation*}
extends to an injective $*$-homomorphism.
\item[\textup{(ii)}] $\dot{V}_a=\dot{V}_b$,
provided that $a, b \in \W$ are neutral words with $K(a)=K(b)\neq \emptyset$.
\end{itemize}
\end{theorem}

\begin{proof}
Verifying that (i) implies (ii) is elementary. Assume that (ii) is holding. It is easy to see that the representation $ \eL\otimes V$ satisfies (T1), (T2) and (T3) and so the map
\begin{equation*}  
w_p \longmapsto \eL_p \otimes V_p,  \,\, p \in P,
\end{equation*}
extends to a representation of Li's algebra $\ca_s(P)$. By composing with the canonical isomorphism of Theorem~\ref{thm;Li} we obtain a representation $\pi$ of $\bTu (P)$ satisfying $\pi(\bt_p) = \eL_p \otimes V_p$, for all $p \in P$. Theorem~\ref{thm;main1} implies now that the map 
\begin{equation}  \label{eq;corFell1}
\eL_p \longmapsto \eL_p \otimes \pi(\bt_p) = \eL_p \otimes \eL_p \otimes V_p,  \,\, p \in P,
\end{equation}
extends to a $*$-isomorphism between the generated $\ca$-algebras.

On the other hand, we have from Corollary~\ref{C:firststep}
that $\eL_p \otimes \eL_p \longmapsto \eL_p$, $p \in P$ extends to a $*$-isomorphism and so by a fundamental property of the spatial tensor product of $\ca$-algebras, the same is true for 
\begin{equation} \label{eq;corFell2}
\eL_p \otimes \eL_p \otimes V_p\longmapsto \eL_p  \otimes V_p, \,\, p \in P.
\end{equation}
The conclusion follows now by composing the $*$-isomorphisms coming from (\ref{eq;corFell1}) and (\ref{eq;corFell2}).
\end{proof}

\begin{remark}
One might hope that the strategy in the proof of Theorem~\ref{thm;main2} might also work for the map
\[
\T_{\la}(P) \ni L_p \longmapsto  L_p \otimes V_p,  \,\, p \in P,
\]
and thus show that it extends to an injective $*$-homomorphism from the Toeplitz algebra $\T_{\la}(P) $.
We remind the reader that this cannot be done even in the simplest case where $V$ is the left regular representation.
\end{remark}

The enhanced left regular representation and the connection with $\ca_s(P)$ of Theorem~\ref{thm;Li} offers great insight in the structure of semigroup $\ca$-algebras. Apart from establishing Theorem~\ref{thm;main2} and Corollary~\ref{cor;question} it will allow us to obtain now direct proofs for results which were obtained elsewhere with more involved proofs. In Section~\ref{Sec;nsa} we will see more applications.

Consider the inverse semigroup $\I_{L}:= \{\dot{L}_a\mid a \in \W\} \subseteq \T_{\la}(P)$ and let $\ca(\I_{L})$ be its universal $\ca$-algebra. The following was obtained in \cite{Li13} under the assumption that $P$ satisfies the Toeplitz condition. That assumption was removed several years later in \cite{KKLL} with a proof that required a clever argument. It now follows directly from Theorem~\ref{thm;Li}.

\begin{corollary}[Theorem 3.2 of \cite{KKLL}] \label{cor;invsem}
Let $P$ be a submonoid of a group $G$. Then $\ca_s(P)$ is canonically isomorphic to $\ca(\I_{L})$.
\end{corollary}

\begin{proof}
Note that every inverse semigroup representation of $\I_{L}$ determines a Li-covariant representation of $P$. Hence there exists a $*$-homomorphism $\rho : \ca_s(P) \rightarrow \ca(\I_{L})$ that sends generators to generators. Now the representation
\[
\dot{L}_a \longmapsto \dot{\eL}_a, \,\, a \in \W,
\]
is clearly an inverse semigroup representation of   $\I_{L}$ and so we have a map $\ca(\I_{L})\rightarrow \ca(\eL)\simeq \ca_s(P) $ who acts as the inverse of $\rho$.
\end{proof}

Here is another application of Theorem~\ref{thm;Li}. Recall that a cancellative semigroup $P$ is said to satisfy independence if for every $X \in \J$ and all $X_1, X_2, \dots, X_n \in \J$, 
\[
X= \cup_{i=1}^n X_i
\]
implies that $X=X_i$ for some $i=1,2,\dots,n$. In \cite[Corollary 3.3]{KKLL} it was shown that if the map
\[
\ca_s(P) \ni w_p \longmapsto L_p \in \T_{\la}(P), \,\, p \in P,
\]
is an isomorphism, then $P$ satisfies independence. The proof is based on \cite[Theorem 3.2]{KKLL} and an earlier result of Norling \cite[Theorem 3.2.14]{Nor14} requiring inverse semigroup theory. Here is a proof that avoids that theory.

\begin{corollary}[Corollary 3.3 of \cite{KKLL}]
Let $P$ be a submonoid of a group $G$. If the map 
\begin{equation} \label{eq;quick}
\ca_s(P)_e \ni \dot{w}_a \longmapsto \dot{L}_a \in \T_{\la}(P)_e, \,\, a \in \W, \dot{a}=e,
\end{equation}
is injective, then $P$ satisfies independence.
\end{corollary}

\begin{proof}
Consider the map
\[
\T_{\la}(P)_e  \ni  \dot{L}_a \longmapsto \dot{w}_a \longmapsto  \dot{\eL}_a \longmapsto  \dot{\eL}_a \otimes  \dot{\eL}_a  \longmapsto  \dot{L}_a\otimes  \dot{L}_a, \,\, a \in \W, \dot{a}=e,
\]
where the first arrow (from the left) comes from the inverse of (\ref{eq;quick}), the second arrow comes from the universality of $\ca_s(P)\simeq \bTu(P)$, the third comes from Corollary~\ref{C:firststep} and the fourth from restricting to a reducing subspace. The conclusion follows now from combining \cite[Proposition 2.24(v)]{Li12} with \cite[Corollary 2.22]{Li12}.
\end{proof}

\section{The non-selfadjoint theory}\label{Sec;nsa}

It is natural to ask how the concepts explored in the previous sections manifest in the non-selfadjoint context. Indeed, consider the non-selfadjoint algebra $\bTl(P)^+$ generated by the image of the enhanced left regular representation of a submonoid $P$. Such an algebra is $\bTl(P)^+$ the ``enhanced" analogue of the familiar tensor algebra $\T_{\la}(P)^+$ which has been studied extensively \cite{ClD, CR19, DK20, DKKLL, KKLL, MS98, MS00}. One might ask how does $\bTl(P)^+$ compared to $\T_{\la}(P)^+$. The following answers that question.

\begin{theorem} \label{prop;identify}
Let $P$ be a submonoid of a group $G$ and let $\bTl (P)^+$ denote the non-selfadjoint algebra generated by the enhanced left regular representations $\eL$. Then $\bTl (P)^+$ is completely isometrically isomorphic to $\T_{\la} (P)^+$ via a map that sends generators to generators.
\end{theorem}

\begin{proof}
Consider the normal coaction 
\[
\de \colon \T_{\la}(P) \longrightarrow \T_\la(P)  \otimes \ca(G) : L_p \longmapsto L_p \otimes u_{p}
\]
appearing in the proof of Proposition~\ref{P:f coaction} and note that the completely contractive  map defined by
\[
\T_{\la}(P)^+  \ni L_p \xmapsto{\phantom{nn} \de \phantom{nn}} L_p \otimes u_p \xmapsto{\phantom{nnnn} \phantom{nnn}} L_p \otimes ( l_p|_{\ell^2(P)})= L_p\otimes L_p
\]
is multiplicative because $\ell^2(P)$ is invariant by all $l_p$, $p \in P$. Iterations of the above argument show that the maps 
\[
\T_{\la}(P)^+\ni L_p \xmapsto{\phantom{nnn} \phantom{nnn}}  L_p^{\otimes n}, \,\, n =3, 4, \dots
\]
are completely contractive and multiplicative. By taking a direct sum of all these maps, we produce a completely isometric map from $\T_{\la}(P)^+$ onto $\bTl (P)^+$, which sends generators to generators.
\end{proof}

The reduced $\ca$-algebra $\ca_{\la}(\I_L)$ associated with the inverse semigroup $\I_L$ (see Corollary~\ref{cor;invsem}) is the $\ca$-algebra generated by the
operators $\fl (w) : \ell^2(\I_L \backslash \{0\}) \rightarrow \ell^2(\I_{L} \backslash \{0\})$  determined by
\[
\fl (w) \de_x = 
\begin{dcases*}
\de_{wx} 
   & if  $w^*w\geq xx^*$ \\[1ex]
0 
   & otherwise,
\end{dcases*}
\]
with $w\in \I_{L}$. By \cite[Theorem 3.10]{KKLL} and Theorem \ref{thm;Li}, we have that $\bTl(P)$ is canonically isomorphic to $\ca_{\la}(\I_L)$. Therefore Theorem \ref{prop;identify} leads to the following alternative description of the tensor algebra $\T_{\la} (P)^+$.

\begin{corollary}
Let $P$ be a submonoid of a group $G$. Then the map 
\[
\T_{\la}  (P)^+ \ni L_p\longrightarrow \fl_p \in \ca_{\la}(\I_L), \,\, p \in P,
\]
extends to a completely isometric representation of $\T_{\la} (P)^+$.
\end{corollary}

We can now use our theory to see that in contrast with its selfadjoint counterpart, $\T_{\la}(P)^+$ admits a semigroup comultiplication for an arbitrary submonoid $P$. This not only extends \cite[Corollary 5.5]{ClD} but also makes the RFD-coaction theory developed in \cite{ClD} available for all submonoids, not just the ones satisfying independence. We will have to say more about this theory shortly.

\begin{corollary} \label{below}
Let $P$ be a submonoid of a group $G$. Then there is a completely isometric comultiplication $\De_P$ on $\T_{\la}(P)^+$ given by $\De_P (L_p)=L_p\otimes L_p$, $p \in P$.
\end{corollary}

\begin{proof}
Let $\psi: \T_\la(P)^+ \rightarrow  \bTl(P)^+$ be the isomorphism of Theorem~\ref{prop;identify} and let 
\[
\phi: \bTl(P) \rightarrow \bTl(P) \otimes \bTl(P);\eL_p\longmapsto \eL_p\otimes \eL_p, \,\, p \in P,
\] be the map of Corollary~\ref{C:firststep}. Then the map $\De_P:= (\psi^{-1}\otimes \psi^{-1})\phi \psi $ is the desired comultiplication.
\end{proof}

Of course we could obtain a more general result than that of Corollary~\ref{below} by invoking our Fell's absorption principle (Theorem~\ref{thm;main2}) instead of Corollary~\ref{C:firststep} in its proof. Instead we follow a different path and we obtain a much stronger result. We begin with Theorem~\ref{thm;char} below in order to characterize for which submonoids $P$, the tensor algebra $\T_{\la}^+(P)$ admits a special kind of character. Apart from its intrinsic interest, this characterization will allow us later on to clarify one of the assumptions of \cite[Theorem 4.6]{ClD} and obtain a non-selfadjoint Fell's absorption principle (Theorem~\ref{thm;main3}). We remark that the requirement below that $P$ is to be contained in an amenable group is not as strict as it appears at first. Many highly non-amenable groups contain as semigroups submonoids of amenable groups. The most distinguished example of such a behavior is  that of the free semigroup on $n$-generators, which is contained in an amenable (actually solvable) group. (See \cite[pg 238]{Lib} for more information.)

\begin{theorem} \label{thm;char}
If $P$ is a submonoid of a group $G$, then the following are equivalent
\begin{itemize} 
\item[\textup{(i)}] $\T_{\la}(P)^+$ admits a character $\omega$ satisfying $\om(L_p)=1$, for all $p \in P$.
\item[\textup{(ii)}] $P$ is left reversible, i.e., $pP\cap qP\neq \emptyset$ for any $p, q \in P$, and it embeds in an amenable group.
\end{itemize}
\end{theorem}

\begin{proof}
Assume first that $\T_{\la}(P)^+$ admits such a character with $\om(L_p)=1$, for all $p \in P$. Extend $\om$ to a positive linear form $\hat{\om}$ on $\T_{\la}(P)$ and notice that for any $p \in P$, we have 
\[
\hat{\om}(L_p^*)\hat{\om}(L_p)= \hat{\om}(L_p^*L_p) =1
\]
and so all $L_p$, $p \in P$, belong to the right multiplicative domain of $\hat{\om}$ (\cite[pg 39]{Pau02}). According to \cite[Theorem 3.18(i)]{Pau02} 
\[
\{a \in \T_{\la}(P)\mid \hat{\om}(a)\hat{\om}(a)= \hat{\om}(a^*a) \}= \{a \in \T_{\la}(P)\mid \hat{\om}(b)\hat{\om}(a)= \hat{\om}(ba), \mbox{ for all } b \in \T_{\la}(P) \}
\]
and so with $a =L_p$ and $b=L_pL_p^*$ in the above, we obtain
\[
1= \hat{\om}(L_p) = \hat{\om}((L_p L_p^* ) L_p) = \hat{\om}(L_p L_p^* )\hat{\om}(L_p) = \hat{\om}(L_p L_p^* ).
\]
Hence all $L_p$, $p \in P$ belong to the left multiplicative domain of $\hat{\om}$. Therefore $\hat{\om}$ is a multiplicative form on $\T_{\la}(P)$ and so $P$ is left amenable. By \cite[Proposition (1.25)]{Pat} $P$ is left reversible and by \cite[Proposition (1.27)]{Pat} $P$ embeds in an amenable group.

Conversely assume that $P$ embeds in an amenable group $G'$. Theorem~\ref{prop;identify} implies that it is enough to show that $\bTl(P)^+$ admits such a character. We will show instead that $\ca_s(P)$ admits a character and this will suffice since the amenability of $G'$ implies that $ \bTl(P)=\ca_s(P)$ and so  $\bTl(P)^+ \subseteq \bTl(P)$ inherits that character.

Consider the semigroup representation $P \ni p \mapsto 1\in \bC$. This representation obviously satisfies the defining relations (T1) and (T3) for $\ca_s(P)$. It also satisfies (T2) because when $P$ is left reversible, then $K(a)\neq \emptyset$ for any non-emty neutral word $a \in \W$. This is a well known fact but we sketch a proof for completeness. 

Indeed it suffices to show that for any word $b \in W$ and $p \in P$, we have $\dot{L_b}L_p\neq 0$ (or equivalently, $\dot{L_b}L_pL_p^*\neq 0$), provided that $\dot{L_b} \neq 0$. Now $\dot{L_b}L_pL_p^*\neq 0$ means that its domain projection has to be non-zero, i.e., $K(b(e,p)(p, e))\neq \emptyset$. Let $q \in K(b)$, which is non-empty by assumption, and so $qP\subseteq K(b)$. But then \cite[Proposition 2.6(5)]{LaS} implies that
\[
K(b(e,p)(p, e))=K(b)\cap K((e,p)(p,e)) \supseteq qP\cap pP\neq \emptyset ,
\]
because $P$ is left reversible and we are done.

Since the semigroup representation $P \ni p \mapsto 1\in \bC$ satisfies all the defining relations of $\ca_s(P)$, it induces a representation on $\bC$, i.e., a character for $\ca_s(P)$, as desired.
\end{proof}

Initial motivation for our Theorem~\ref{thm;char} came from one of the main results of \cite{ClD}, Theorem~4.6, which has as one of its assumptions that the tensor algebra $\T_{\la}(P)^+$ admits a character $\omega$ satisfying $\om(L_p)=1$, for all $p \in P$. Theorem~\ref{thm;char} clarifies now that the presence of such a character $\om$ is equivalent to $P$ being left reversible and embedding in an amenable group. Here is a more significant application of Theorem~\ref{thm;char}.

\begin{theorem}[Fell's absorption principle for semigroups; non-selfadjoint version] \label{thm;main3}
Let $P$ be a submonoid of a group $G$, let $L$ be the left regular representation of $P$ and let $V= \{ V_p\}_{p \in P}$ be a representation of $P$ by isometries on Hilbert space $\H$. Consider
\begin{itemize}
\item[\textup{(i)}] The representation $V= \{ V_p\}_{p \in P}$ extends to a completely contractive representation of $\T_{\la}(P)^+$.
\item[\textup{(ii)}] The map
\begin{equation*} 
\T_{\la}(P)^+ \ni L_p \longmapsto L_p \otimes V_p,  \,\, p \in P,
\end{equation*}
extends to a completely isometric representation of $\T_{\la}(P)^+$.
\end{itemize}
Then \textup{(i)} $\implies $ \textup{(ii)}. If $P$ is left reversible and $G$ amenable, then  \textup{(ii)} $\implies$  \textup{(i)}.
\end{theorem}

\begin{proof}
Assume that we have a completely contractive representation 
\[
\pi_{V} \colon \T_{\la}(P)^+\longrightarrow B(\H): \, L_p\longmapsto  V_p, \,\, p \in P.
\]
If $\Delta_P$ is the comultiplication of Corollary~\ref{below}, then the completely contractive map $\phi:=(\id\otimes \pi_V)\circ\Delta_P$ indeed satisfies $\phi(L_p)= L_p \otimes V_p$, for all $p \in P$. 
We need to verify now that $\phi$ is completely isometric, not just a complete contraction. For that, we use a familiar trick.

Consider the isometry 
\[
W \colon \ell^2(P)\otimes \H \longrightarrow \ell^2(P)\otimes \H;\,   \de_p \otimes V_qh \longmapsto   \de_p \otimes V_{pq}h, \,\, p,q \in P, h \in \H,
\]
where $\{\de_p\}_{p \in P}$ is the canonical orthonormal basis of $\ell^2(P)$. It is easy to verify that 
\[
W(L_p\otimes I)=(L_p\otimes V_p)W, \,\, \mbox{for all } p \in P,
\]
by applying the left and right sides of the above equation on a vector of the form $\de_s \otimes V_t h$. Therefore
\begin{equation*}
\begin{split} 
\|L_p\|&= \|W(L_p\otimes I)\|  \\
		&=\|(L_p\otimes V_p)W\| \\
		&\leq \| L_p\otimes V_p\| =\|\phi(L_p)\|
	\end{split}
\end{equation*}
and so $\phi$ is an isometry. A matricial analogue of the above argument establishes that $\phi$ is  a complete isometry and we are done.

Assume now that $P$ is left reversible, $G$ is amenable and that the map
\begin{equation*} 
\T_{\la}(P)^+ \ni L_p \longmapsto L_p \otimes V_p,  \,\, p \in P,
\end{equation*}
extends to a completely isometric representation $\phi$ of $\T_{\la}(P)^+$. Theorem~\ref{thm;char} implies now the existence of 
a character $\omega$ satisfying $\om(L_p)=1$, for all $p \in P$. Then the map $(\om\otimes \id)\phi$ is a completely contractive representation of $\T_{\la}(P)^+$ satisfying 
\[
(\om\otimes \id)\phi(L_p)= 1\otimes V_p \simeq V_p, \mbox{ for all } p \in P,
\]
and the conclusion follows.
\end{proof}

One can recast Theorem~\ref{thm;main3} using the language of semigroup coactions of Clouatre and Dor-On \cite{ClD}.

\begin{definition}[Clouatre and Dor-On \cite{ClD}]
Let $A$ be an operator algebra, $P$ be a cancellative discrete semigroup and let $\T_{\la}^+(P)$ denote the non-selfadjoint algebra generated by the left regular representation of $P$. A completely isometric homomorphism $\de: A \rightarrow A\otimes \T_{\la}^+(P)$ is said to be a coaction of $P$ on $A$ if the linear span of the spectral subspaces 
\[
A_p:= \{ a \in A\mid \de(a)=a\otimes L_p\}, \,\, p \in P,
\]
is norm dense in $A$.
\end{definition}

In light of the above definition, our  Theorem~\ref{thm;main3} characterizes the semigroup representations of $P$ that lead to coactions of $P$ on $\T_{\la}(P)^+$.  In \cite{ClD}, Clouatre and Dor-On make heavy use of these semigroup coactions in order to investigate questions regarding residual finite dimensionality in the context of general operator algebras. In particular they are interested when the maximal $\ca$-cover $\cmax (\A)$ of an operator algebra $\A$ is residually finite dimensional (RFD). All of their results are confined however to semigroups which are independent, the reason being that they utilize heavily the selfadjoint theory and there, as we have already mentioned, there are significant issues with the left regular representation beyond independent semigroups. 

The following result was obtained by Clouatre and Dor-On \cite{ClD} only in the case where $P$ is independent and it required an intricate proof. 

\begin{theorem} \label{thm;ClDeasy}
Let $P$ and $Q$ be submonoids of groups $G$ and $H$ respectively. Let $\phi: G\rightarrow H$ be a group homomorphism such that $\phi(P) \subseteq Q$. Then the map 
\begin{equation} \label{eq;preco}
\T_{\la}(P)^+ \ni L_p^P \longmapsto L_p^P\otimes L_{\phi(p)}^Q\in \T_{\la}(P)^+\otimes  \T_{\la}(Q)^+
\end{equation}
extends to a coaction of $Q$ on $\T_{\la}(P)^+$.
\end{theorem}

\begin{proof} 
First notice that the group homomorphism $\phi: G\rightarrow H$ induces now a $*$-homo-\break morphism 
\[
\pi_{\phi}: \ca(G) \rightarrow \ca(H); u_g \longmapsto u_{\phi(g)}
\]
Consider again the normal coaction 
\[
\de \colon \T_{\la}(P) \longrightarrow \T_\la(P)  \otimes \ca(G) : L_p^P \longmapsto L_p^P \otimes u_{p}
\]
appearing in the proof of Proposition~\ref{P:f coaction} and note that the completely contractive  map $\de_{\phi}$ defined on $\T_{\la}(P)^+$ by 
\[
\T_{\la}(P)^+\ni L_p^P \xmapsto{\phantom{nn} \de \phantom{nn}} L_p^P \otimes u_p \xmapsto{\phantom{nn} \id \otimes \pi_{\phi} \phantom{nn}} L_p^P\otimes u_{\phi(g)} \xmapsto{\phantom{nn} \phantom{nn}} L_p^P  \otimes ( l_{\phi(p)}|_{\ell^2(Q)})= L_p^P\otimes L_{\phi(p)}^Q
\]
is multiplicative because $\ell^2(Q) \subseteq \ell^2(H)$ is invariant by all $l_q \in \ca_r(H)$, $q \in Q$. Since each $\de_{\phi}(L_p)$, $p \in P$, is an isometry, Theorem~\ref{thm;main3} implies that the map 
\begin{equation*}
\T_{\la}(P)^+\ni L_p^P \xmapsto{\phantom{nmn}}  L_p^P\otimes \de_{\phi}(L_p) = L_p^P\otimes L_p^P \otimes L_{\phi(p)}, \,\, p \in P,
\end{equation*}
extends to a completely isometric representation of $\T_{\la}(P)^+$. However Corollary~\ref{below} shows that the map 
\[
\T_{\la}(P)^+\otimes \T_{\la}(P)^+\  \ni L^P_p  \otimes L^P_p  \ \longmapsto L^P_p , \,\,  p \in P,
\]
extends to a completely isometric map and the conclusion follows by composing the maps above.
\end{proof}

We can now obtain a strengthening of one of the main results in \cite{ClD}. Let $P$ and $Q$ be submonoids of groups $G$ and $H$ respectively. Clouatre and Dor-On define a homomorphism $\phi :G \rightarrow H$ as a $(P,Q)$-map if $\phi(P)\subseteq Q$ and the set $\phi^{-1}(q)\cap P$ is finite for any $q \in Q$. Maps with similar properties have been considered many times before, starting with \cite{LR96} where they were called controlled maps. 

Recall that a cancellative semigroup P is said to have the finite divisor property (or FDP) if each element of $P$ admits finitely many factorizations. 

\begin{theorem} \label{thm;impr}
Let $P$ and $Q$ be countable submonoids of groups $G$ and $H$, with $H$ amenable. Assume that $Q$ is left reversible and has FDP. If there exists a $(P, Q)$-map from $G$ to $H$, then $\cmax(\T_{\la}(P)^+)$ is residually finite dimensional.
\end{theorem}

\begin{proof}
The proof is identical to that of \cite[Theorem 5.7]{ClD} if one instead of using \cite[Theorem 5.4]{ClD} in their proof now one uses our Theorem~\ref{thm;ClDeasy}.
\end{proof}

Theorem~\ref{thm;impr} improves \cite[Theorem 5.7]{ClD}, one of the main results of that paper, in two ways. First, it removes the assumption that the semigroups involved are independent. Furthermore, \cite[Theorem 5.7]{ClD} asks that the semigroup $Q$ is left amenable. As it can be seen in \cite[Theorem 5.5.42]{Lib} and its proof, left amenable semigroups embed in amenable groups. Furthermore, their Toeplitz algebra has to be nuclear. By asking $P$ to be merely left reversible we dispose of the nuclearity requirement. 

Our Theorem~\ref{thm;impr} allows us to obtain the following improvement of \cite[Corollary 5.8]{ClD}.

\begin{corollary} \label{cor;RFD}
Let $P$ be a submonoid of an amenable group $G$. If $P$ is left reversible and has FDP, then $\cmax(\T_{\la}(P)^+)$ is RFD.
\end{corollary}

Using the results of this paper and ideas from \cite{DKKLL} we also obtain yet another (perhaps more transparent) proof for the existence a co-universal $\ca$-algebra for equivariant, Li covariant representations. The real bonus here is that on our way we obtain Theorem~\ref{P:P cover} which provides a significant generalization of the important Theorem~\ref{prop;identify}.

If $\A$ is an operator algebra and $\de \colon \A \to \A \otimes \ca(G)$ is a  coaction on $\A$, then we will refer to the triple $(\A, G, \de)$ as a  \emph{cosystem}. 
A map $\phi \colon \A \to \A'$ between two cosystems $(\A, G, \de)$ and $(\A', G, \de')$ is said to be \emph{$G$-equivariant}, or simply \emph{equivariant}, if $\de' \phi=(\phi\otimes \id)\de$.

\begin{definition}\label{D:coaction}
Let $(\A, G, \de)$ be a cosystem.
A triple $(C', \iota', \de')$ is called a \emph{C*-cover} for $(\A, G, \de)$ if  $(C', G, \de')$ forms a cosystem and $(C', \iota')$ forms a C*-cover of $\A$ with $\iota : \A\rightarrow C'$ being equivariant.
\end{definition}

\begin{definition}
Let $(\A, G, \de)$ be a cosystem.
The \emph{C*-envelope of $(\A, G, \de)$} is a C*-cover for $(\A, G, \de)$, denoted  by $( \cenv(\A, G, \de), \iotenv, \delenv)$,
that satisfies the following property: 
for any other C*-cover $(C', \iota', \de')$ of  $(\A, G, \de)$ there exists an equivariant  $*$-epimorphism $\phi \colon C' \to \cenv(\A, G, \de)$ that makes the following diagram
\[
\xymatrix{
& & C' \ar@{.>}[d]^{\phi} \\
\A \ar[rru]^{\iota'} \ar[rr]^{\iotenv} & & \cenv(\A, G, \de)
}
\]
commutative. We will often omit the embedding $\iotenv$ and the coaction $\delenv$  and  refer to the triple simply  as $ \cenv(\A, G, \de)$.
\end{definition}

As in the case of the C*-envelope for an operator algebra, it is easily seen that if the C*-envelope for a cosystem exists, then it is unique up to a natural notion of  isomorphism for cosystems.
The existence of the C*-envelope for an arbitrary cosystem was established in \cite[Theorem 3.8]{DKKLL}, where it was used as a means to solve an open problem from \cite{CLSV11}.  For information regarding $\ca$-envelopes and their theory, the reader should consult \cite{BL04, Pau02}. The origins of the theory go back to the seminal work of Arveson from the late 60's \cite{Arv69, Arv72}.

In Theorem~\ref{prop;identify} we saw that $\bTl(P)$ contains canonically a completely isometric copy of the tensor algebra $\T_{\la}(P)^+$. Our next result shows that a variety of other $\ca$-algebras may also contain a copy of the tensor algebra $\T_{\la}(P)^+ \simeq \bTl(P)^+$.

\begin{theorem}\label{P:P cover}
Let $P$ be a submonoid of a group $G$. 
Let $(B, G, \de)$ be a cosystem for which there exists an equivariant epimorphism
\[
\phi \colon \bTl (P) \longrightarrow B.
\]
Then $\phi|_{\bTl (P) ^+}$ is completely isometric and therefore $(B, G, \de)$ forms a C*-cover for the cosystem $(\bTl (P) ^+, G, \dep)$.
\end{theorem}

\begin{proof}
Consider the unital completely positive map
\[
\psi \colon \ca(G) \longrightarrow \ca_r(G) \longrightarrow \B(\bar{\ell}^2(P)) : u_g \mapsto \enl_g \mapsto P_{\eel} \enl_g |_{\eel}
\]
which is multiplicative on the subalgebra of $\ca(G)$ generated by all $u_p$, $p \in P$. Since $\psi(u_p) = \eL_p$ for all $p\in P$, the following diagram of completely contractive homomorphisms
\[
\xymatrix{
\bTl (P)^+ \ar[d]_{\phi} \ar[rr] & & B \otimes \bTl (P) \\
B \ar[rr] & & B \otimes \ca(G) \ar[u]_{\id \otimes \psi}
}
\]
commutes.
By Theorem~\ref{thm;main1} the upper horizontal map is a restriction of an injective $*$-homomorphism and thus it is completely isometric.
Hence $\phi$ is completely isometric.
\end{proof}

Note that we could have obtained Theorem~\ref{prop;identify} as a corollary of Theorem~\ref{P:P cover} by considering the equivariant epimorphism 
\[
\bTl(P) \ni \eL_p\longmapsto L_p \in \T_{\la}(P), \,\, p \in P.
\]
However, the current proof is more direct as it does not require the use of our Fell's absorption principle (Theorem~\ref{thm;main1}).

\begin{definition}
We say that a Li-covariant representation $v$ of a submonoid $G$ is  \emph{gauge-compatible}, or simply \emph{equivariant} if $\ca(v)$ admits a coaction of $G$  that makes the canonical epimorphism $\bTu (P) \rightarrow \ca(v)$ equivariant with respect to the natural  (gauge) coaction of $G$ on $\bTu (P) $. 
\end{definition}

\begin{definition}\label{D:couniversal}
 Let $P$ be a submonoid of a group $G$.
Suppose  $(\couni, G,\gamma)$ is a cosystem and $j:P \to \couni$ is a Li-covariant isometric representation, with integrated version denoted by $j_*: \ca_s(P) \to \couni$.
We say that  $(\couni, G, \gamma, j)$ has the \emph{co-universal property for equivariant, Li-covariant  representations of $P$} if
\begin{enumerate}
\item $j_*: \ca_s(P) \to \couni$ is  $\hat{\de}$-$\gamma$ equivariant; and 
\item for every  equivariant, Li-covariant representation $v: P \rightarrow \ca(v)$, 
there is a surjective $*$-homomorphism $\phi : \ca (v) \rightarrow \couni$ such that $$\phi(v_p) = j_p, \mbox{ for all }  p \in P.$$
\end{enumerate}
Notice that, as observed at the beginning of \cite[Section 4]{CLSV11},  the map $\phi$ is automatically equivariant because $j_*$ and $t_*$ are surjective.
\end{definition}

Our next result shows that the C*-envelope of the tensor algebra $\T_\la(P)^+$ taken with its  natural coaction satisfies the co-universal property. This result was first obtained in \cite{KKLL} and later was generalized in \cite{Seh}. Both proofs make use of strong covariance, as developed in \cite{Seh18}. Our proof below is self-contained and perhaps more transparent.

\begin{theorem} \label{T:co-univ}
Let $P$ be a submonoid of a group $G$.
Let 
$\dep \colon \bTl (P)^+ \rightarrow \bTl (P)^+ \otimes \ca(G)$ be
the restriction of the coaction from Proposition \ref{P:f coaction} to $\bTl (P)^+$. Then the C*-envelope  
$
 (\cenv(\bTl (P)^+, G, {\dep}),\delenv,\iotenv)
$
of the cosystem $(\bTl (P)^+, G, {\dep})$ 
 satisfies the co-universal property associated with equivariant, Li-covariant  representations of $P$.
\end{theorem}

\begin{proof}
By definition $(\cenv(\bTl (P)^+, G, {\dep}),\delenv,\iotenv)$ is generated by a Li-covariant, $G$-compatible representation of $P$.
It remains to show that it has the required co-universal property.

Let $\ol{E}$ be the faithful conditional expectation on $\bTl (P)$ and let $$\bl \colon \ca_s(P) \simeq \bTu (P) \to \bTl (P)$$ be the enhanced left regular representation.
Then by \cite[Proposition 3.6]{Exe97} we have 
\[
\ker \bl  = \{a\in \bTu (P) \mid \ol{E}(a^*a) = 0\}.
\]

Let $v$ be  an equivariant, Li-covariant representation of $P$.
Then $\ca(v)$ admits a $G$-grading and let us write $\B = \{B_g\}_{g \in G}$ for this grading of $\ca(v)$.
Due to the existence of the conditional expectation on $\ca(v)$, by \cite[Theorem 3.3]{Exe97}  there exists a canonical equivariant $*$-epimorphism
\[
\phi \colon \ca_s (P)  \longrightarrow \ca(v) \longrightarrow \ca_r(\B)
\]
where $\ca_r(\B)$ is the reduced cross sectional  C*-algebra of the Fell bundle $\B$. If $E'$ is the associated faithful conditional expectation on $\ca_r(\B)$, then $\phi$ intertwines $\ol{E}$ and $E'$ due to its equivarience.

Let $a \in \ker \bl$. As $\phi$ intertwines the conditional expectations  $\ol{E}$ and $E'$, we derive that $E'(\phi(a^*a)) = 0$ and so $\phi(a) = 0$, because $E'$ is faithful.
Since $a$ was arbitrary in $\ker \bl$ we get that $\ker \bl \subseteq \ker \phi$.
Hence there is an induced $*$-homomorphism $\phi'$ that makes the following diagram
\[
\xymatrix{
\ca_s (P)\ar[rr]^{\phi} \ar[rd]^{\bl} & & \ca_r(\B) \\
& \bTl (P) \ar@{.>}[ur]^{\phi'} &
}
\]
commutative.
By construction $\phi'$ is equivariant and so Proposition \ref{P:P cover} implies that $\ca_\la(\B)$ is a C*-cover of $(\bTl (P)^+, G, \dep)$.
Therefore we have the following $*$-epimorphisms
\[
\ca(v) \longrightarrow \ca_r (\B) \longrightarrow \cenv(\bTl (P)^+, G, {\dep),}\delenv,\iotenv),
\]
which establishes that  $ \cenv(\bTl (P)^+, G, {\dep),}\delenv,\iotenv)$ satisfies the  co-universal property for $P$, as desired. 
\end{proof}

\section{Concluding remarks and open problems}
\noindent (i) In this paper we advocate for an alternative approach in the study of semigroups and their Hilbert space representation theory. An approach which instead of having the left regular representation as a centerpiece, it focuses on the enhanced left regular representation. Such an approach is of course dedicated to semigroups which do not satisfy independence. Therefore it seems fitting to present a very accessible  example of a submonoid that does not satisfy independence for the benefit of the uninitiated reader. The example is taken from \cite{Lib}, where a variety of additional examples can be found. Additional examples can be found in \cite[Section 9.2]{LaS}.

\begin{example} Consider the additive semigroup $P:=\bN \backslash\{1\}$. Following \cite[pg. 229]{Lib} we show that $P$ does not satisfy independence. 

Indeed, consider the constructible ideals 
\[
k+P= L_kL_k^* =\{k , k+2, k+3, \dots\}, \,\, k= 2, 3, \dots.
\]
Then
\[
5+\bN= (2+P)\cap(3 +P)
\]
is also a constructible ideal. However $5+\bN = (5+P)\cup (6+P)$, and yet $5+\bN \neq k +P$, for any 
 $k= 2, 3, \dots$. Hence $P$ does not satisfy independence.
 
 Actually, a similar argument shows that any numerical semigroup of the form $\bN\backslash F$ with $F \subseteq \bN$ finite, fails the independent condition.
 \end{example}

\noindent(ii) We would like to know the answer to the following problem

\begin{problem}
Let $P$ be a submonoid of a group $G$, let $L$ be the left regular representation of $P$ and let $V= \{ V_p\}_{p \in P}$ be a representation of $P$ by isometries on Hilbert space $\H$. Characterize the representations $V$ for which $L_p \mapsto L_p \otimes V_p$ extends to an isometric representation of $\T_{\la}(P)^+$.
\end{problem}

Theorem~\ref{thm;main3} gives an answer when $P$ is a left reversible submonoid of an amenable group. We are wondering what happens beyond that case.

\begin{acknow}
Elias Katsoulis was partially supported by the NSF grant DMS-2054781.
\end{acknow}


\end{document}